\documentclass{article}
\usepackage[left=4cm,top=3cm,right=3cm,bottom=3cm]{geometry}
\usepackage[english]{babel}
\usepackage{setspace} 
\usepackage{latexsym}
\usepackage{amssymb,amsmath,amstext,amsthm,amsfonts}
\usepackage[latin1]{inputenc}
\usepackage{amsmath}
\usepackage{amsthm}
\usepackage{color}

\input pictex.tex

\newcommand{\Z}{\mathbb Z}

\newcommand{\C}{{\mathbb C}^\times}
\newcommand{\F}{{\mathbb F}}
\newcommand{\findem}{\qed}

\theoremstyle{definition}
\DeclareMathOperator{\Sp}{Sp}
\DeclareMathOperator{\SLc}{SL}
\DeclareMathOperator{\GL}{GL}
\DeclareMathOperator{\ot}{O}

\DeclareMathOperator{\uni}{U}

\DeclareMathOperator{\m}{M}
\DeclareMathOperator{\tr}{t}
\DeclareMathOperator{\sign}{sgn}
\DeclareMathOperator{\ide}{id}

\DeclareMathOperator{\Le}{L}
\DeclareMathOperator{\Hom}{Hom}
\DeclareMathOperator{\Aut}{Aut}
\DeclareMathOperator{\Nor}{N}

\newtheorem{thm}{Theorem}[section]

\newtheorem{prop}[thm]{Proposition}

\newtheorem{lem}[thm]{Lemma}

\newtheorem{rmk}[thm]{Remark}

\newtheorem{defi}[thm]{Definition}

\newtheorem{cor}[thm]{Corollary}

\title{ Weil representations  of $\uni(n,n)\left(\F_{q^2}/\F_q\right)$, $q>3$ odd via presentation and compatibility  of methods. }
\author{Luis Guti\'{e}rrez Frez\\ Universidad Austral de Chile, Chile \thanks{The author was  partially supported  by Direcci\'on de Investigaci\'on y Desarrollo, Universidad Austral de Chile, Chile.}\\
 Andrea Vera-Gajardo\\ Durham University, United Kingdom \thanks{The author was  partially supported  by Conicyt-Chile and Durham University, UK.}}
\date{}
\begin{document}
\maketitle

\begin{abstract}
In this article we construct  Weil representations of  quasi-split unitary groups  $\uni(n,n)(\F_{q^2}/\F_q)$ ($q>3$) attached to  quadratic extensions of finite fields,  by using Bruhat-like presentations of those groups. More precisely, we define Weil representations of $\uni(n,n)(\F_{q^2}/\F_q)$ associating to each generator a linear map of a suitable $\mathbb{C}$-vector space satisfying the relations of the aforementioned presentation.
In addition, we also address the natural question on the compatibility of our representation of $\uni(n,n)\left(\F_{q^2}/\F_q\right)$ with the classical one constructed by G\'erardin.
\end{abstract}


\section{Introduction}

Nowadays, Weil representations are a vast topic in mathematical 
research. They arose from A. Weil's celebrated paper \cite{weil} as projective representations of the symplectic groups $\Sp(2n,F)$, where $F$ is a locally compact field. In this work Weil took advantage of the representation theory of the Heisenberg group $\mathcal{H}_n$ associated to a symplectic space of rank $n$, described by the Stone-von Neumann theorem for the real case.
Weil representations have multiple consequences in diverse areas, including number theory, physics  and algebra. A remarkable fact is that by decomposition into irreducible components, they provide a fruitful method for constructing representations of some classical groups over a finite or local field of residual characteristic different from two \cite{ku}. In the same way, in \cite{nobs1}, \cite{nobs2},  \cite{nobs3} the authors use Weil representations to give a description of the representations of $\GL_2(\Z_p)$ and $\SLc_2(\Z_p)$, including $p=2$.\\

Regarding unitary groups,  G\'erardin \cite{ge} constructed  Weil representations of  unitary groups associated to  unitary spaces $(F,i)$. More precisely,  given a quadratic extension $K/k$ of finite fields and  $^{-}$ the nontrivial $k$-automorphism of $K$, we let $(F,i)$ denote a $K$-vector space $F$ equipped with the nondegenerate skew-hermitian form $i$ given by $^{-}$.  Let $E$ be the underlying $k$-vector space of $F$. G\'erardin attached to  $(F,i)$  a nondegenerate symplectic form $j$ on $E$ in such a way that the unitary group $\uni(F,i)$ is embedded into the symplectic group $\Sp(E,j)$. Thus, he obtained a representation of  $\uni(F,i)$ by restricting the Weil representation of $\Sp(E,j)$ to the subgroup $\uni(F,i)$.\\

 In the seventies, Cartier observed that Weil representations  of $\Sp(2,\mathbb{R})$ could be simply constructed by assigning to each generator of certain presentation of $\SLc(2,\mathbb{R})$ a linear operator of $\Le^2(\mathbb{R})$ satisfying the relations of the given presentation. Then Soto-Andrade \cite{soto} considered this approach to construct a Weil representation of $\Sp(2n,F)$, where $F$ is a finite field. For this purpose, he found a suitable presentation and the corresponding linear operators,
regarding the symplectic group $\Sp(2n,F)$ as a sort of generalized special group $\SLc(2)$,

 Soto-Andrade's idea of considering $\Sp(2n,F)$ as an analogue of $\SLc(2)$ was the seed for a subsequent work in which Pantoja and Soto-Andrade \cite{jofal,comm} have used this approach for a more general setting. More precisely, they defined the groups  $\GL_*^{\varepsilon}(2,A)$ and $\SLc_*^{\varepsilon}(2,A)$ which are groups of matrices over an involutive ring $(A,*)$ whose entries satisfy certain commutation relations involving the involution $*$ and $\varepsilon=\pm 1$. For instance, if we take the full matrix ring $A=\m(n,F)$ and the involution $*$ given by matrix transposition then the corresponding group $\SLc_*^{-1}(2,A)$ is the symplectic group  $\Sp(2n,F)$ and  $\SLc_*^{1}(2,A)$ is the split orthogonal group $\ot(n,n)(F)$. \\

Later on, Guti\'errez Frez, Pantoja and Soto-Andrade \cite{gpsa_1} 
gave a method to construct a very general Weil representation for $\SLc_*^{\varepsilon}(2,A)$ when this group has a Bruhat-like presentation.
 Firstly, they define an abstract core data consisting of a module $M$ over the involutive ring $(A, *)$ equipped with a suitable
nondegenerate complex valued self pairing $\chi$ and its second order homogeneous companion $\gamma$. Then the authors  associated to each generator of $\SLc_*^{\varepsilon}(2,A)$ a  suitable linear operator of $\mathbb{C}^M$ involving the aforementioned data. Finally, they  proved that those operators satisfy the commutating relations of the Bruhat-like presentation of  $\SLc_*^{\varepsilon}(2,A)$.\\.


 Since their origin Weil representations keep receiving much attention. Indeed, there are several papers concerning this topic. For instance, Tanaka \cite{Tan_1,Tan_2} constructed all irreducible  representations of 
  $\SLc_2(\mathbb{Z}/p^k\mathbb{Z})$ by using Weil representations, Soto-Andrade \cite{soto}  obtained all the irreducible representations of $\Sp(4,F)$, by decomposing both Weil representations associated to the two isomorphy types of quadratic forms over $F$ of rank 4, besides the previously mentioned Gerardin's work \cite{ge}.

 Some recent developments have been accomplished by Aubert and Przebinda \cite{aubert}, who gave a wide and detailed description  of the Weil representation attached to a symplectic group over a finite or local field, providing explicit computations and formulae. Herman and Szechtman \cite{He-Sze} constructed  Weil representations of unitary groups attached to ramified quadratic extension of local rings (odd characteristic) by embedding these groups into a symplectic group. Dutta and Prasad \cite{Pr-Du} defined a Weil representation associated to a finite abelian group of odd order, proving that it is multiplicity free and that each irreducible component is associated to an element of a partially ordered set.\\
  
 Regarding Weil representations via presentations, we may mention Vera-Gajardo \cite{vera} constructed a generalized Weil representation for the split orthogonal group $\ot(2n,2n) $ over a finite field $\F_q$ with $q>3$, proving its compatibility with Howe's theory of dual pairs. Guti\'errez Frez \cite{lucho} constructed a generalized Weil representation of $\SLc_*^{\varepsilon}(2,A)$, where $A$ is the truncated polynomial ring $\mathbb{F}_q[x]/\langle x^m\rangle$ equipped with the $\mathbb{F}_q$-linear  involution $x\mapsto -x$ for $\varepsilon=-1$, whereas Guti\'errez Frez and Pantoja \cite{gp} studied the case $\varepsilon=1$.  \\ \\
 
Let $\uni(n,n)(\F_{q^2}/\F_q)$ denote  the group of isometries of a nondegenerate hermitian bilinear form on a $\F_{q^2}$-vector space of dimension $2n$. In this article,  we construct a generalized Weil representation for the groups  $\uni(n,n)(\F_{q^2}/\F_q)$ for $q$ odd greater than $3$, regarding them as an $\SLc_*^{\varepsilon}(2,A)$ group for a suitable unitary ring $A$. For this purpose we will prove that the group has a  Bruhat-like presentation, and then we will construct some data which allow us to obtain the desired representation.
 



The main results of the paper are summarized  by the following theorems:


\textbf{Theorem 3.1.} The quasi-split unitary group $\uni(n,n)\left(\F_{q^2}/\F_q\right)$ is isomorphic to  $\SLc_*^{-1}(2,A_n)$, where $A_n$ is the full matrix ring $\m_{n}(\F_{q^2})$ equipped with the involution given by $(a^*)_{ij}=\overline{a_{ji}}.$


According to  \cite{manusc} Theorem 15, the group $\SLc_*^{-1}(2,A_n)$ has a Bruhat-like presentation. Therefore  this fact and the result above  will allow us to construct a generalized Weil representation for the group. In Theorem \ref{data}, we construct a data  $(M,\chi,\gamma,-1/q^n)$  for  $\SLc_*^{-1}(2,A_n)$ and  by using Theorem 4.4 in \cite{gpsa_1} we prove the following result.



 \textbf{Theorem 4.4.}
There exists a linear (complex) representation  $\left(\Aut(\mathbb{C}^{M}),\rho\right)$   of $\SLc_*^{-1}(2,A_n)$  defined on the generators as follows: 
\begin{itemize}
\item $\rho_{u_b}(e_a)=\gamma(b,a)e_a$,
\item $\rho_{h_t}(e_a)=\alpha(t)e_{at^{-1}}$,
\item $\rho_w(e_a)=c\displaystyle{\sum_{b\in M}\chi(-a,b)e_b}$,
 \end{itemize}
which  we call  generalized  Weil representation of $\SLc_*^{-1}(2,A)\cong \uni(n,n)\left(\F_{q^2}/\F_q\right)$ associated to the data $(M,\chi,\gamma,-1/q^n)$.

Furthermore we address the natural question on the compatibility of our representation with the classical one given by G\'eradin \cite{ge}. In fact, we prove:\\

\textbf{Theorem 6.1.}
The restriction of the classical Weil representation of the symplectic group $\Sp(E,j)$ to the group $\uni(n,n)\left(\F_{q^2}/\F_q\right)$ is the representation $(\mathbb{C}^M,\rho)$ constructed in Theorem 1.2.

Specifically, the article is organized as follows. In section 2, we introduce the main concepts on generalized classical groups   $\GL_*^{\varepsilon}(2,A)$, $\SLc_*^{\varepsilon}(2,A)$ and a construction of Weil representations of $\SLc_*^{\varepsilon}(2,A)$ attached to a data.  In section 3, we verify that  the quasi-split unitary group $\uni(n,n)\left(\F_{q^2}/\F_q\right)$  has a  Bruhat-like presentation. In section 4, we construct a data for the group $\uni(n,n)\left(\F_{q^2}/\F_q\right)$ and we write the linear operator associated to each generator explicitly. In section 5, we give an initial decomposition of the constructed representation. Finally in section 6, we prove the compatibility of methods for the group $\uni(n,n)\left(\F_{q^2}/\F_q\right)$ and we show that the components of the initial decomposition are irreducible.
\vspace{1 cm}

\section{ Preliminaries }
In this section, we shall introduce  the concepts of generalized classical groups   $\GL_*^{\varepsilon}(2,A)$  and $\SLc_*^{\varepsilon}(2,A)$ attached to a unitary  ring $A$ endowed with an involution $*$.  In addition, we also recall the method in \cite{gpsa_1} to construct generalized Weil representations  of  $\SLc_*^{\varepsilon}(2,A)$ when a Bruhat-like presentation and a data for the group are provided.
\subsection{ Generalized Classical Groups}

\vspace{0.4 cm}
Let $A$ be a unitary ring with  an involution $a\mapsto a^*$, {\it i.e.} an antiautomorphism of the ring $A$ of order two. Let $A^{\times}$ be the group of invertible elements of $A$, $Z(A)$ the center of $A$ and $A^{\varepsilon\text{-}sym}$  the set of all $\varepsilon$-symmetric elements of $A$, that is
 $A^{\varepsilon\text{-}sym}= \{a\in A \mid a^*=-\varepsilon a \}$.
The involution on $A$  induces an involution on the full matrix ring $\m(2,A)$, namely  $(g^{*})_{ij} = (g_{ji})^{*}$.
For $\varepsilon=1$ or $\varepsilon=-1$ in $A$ we set  
\[
J = \left(\begin{matrix} 0 & 1 \\\varepsilon 1& 0 \end{matrix}\right) \in \m(2,A).
\]
We write $ML_*^{\varepsilon}(2, A)$ to denote the set of matrices in $\m(2,A)$ such that $g^{*}JgJ^{-1} = \delta(g)I $, for some  $\delta(g)\in Z(A)\cap A^{\varepsilon\text{-}sym}  $. In analogy to the classical case,  Pantoja and Soto-Andrade \cite{comm}  define a  $*$-determinant  function by $\det_*(g) = ad^{*} +\varepsilon bc^{*}$ for $g= \left(\begin{smallmatrix} a & b \\c & d \end{smallmatrix}\right)$.  Then they  proved the following result.
\begin{prop} 
The set $\GL_*^{\varepsilon}(2,A)$, consisting of all invertible elements in $ML_*^{\varepsilon}(2, A)$,  is a group under matrix multiplication and $\det_*$ is an epimorphism of $\GL_*^{\varepsilon}(2,A)$ onto the group of all central $\varepsilon $-symmetric invertible elements of $A$. 
\end{prop}

\begin{defi}
 The group  $\SLc_*^{\varepsilon}(2,A)$ is the kernel of the epimorphism $\det_{*}$.
\end{defi}

\begin{rmk}
Notice that $\SLc_*^{\varepsilon}(2,A)$  is the group of matrices 
$$g=\left(\begin{matrix} a & b \\c& d \end{matrix}\right)$$
with entries in $A$   that satisfy  the following equalities: $a^*c=-\varepsilon c^*a$, $ab^*=-\varepsilon ba^*$, $b^*d=-\varepsilon d^*b$, 
 $cd^*=-\varepsilon dc^*$ and $ad^*+\varepsilon bc^*=a^*d+\varepsilon c^*b=1$.
 


\end{rmk}

\subsection{Weil Representation for $\SLc_*^{\varepsilon}(2, A)$}

Let $A$ be a unitary ring with an involution $*$.
We set
\[
h_t=\left(\begin{array}{cc}
t & 0\\
0 & t^{* -1}\end{array} \right) (t \in A^{\times}),\quad  w= \left(\begin{array}{cc}
0 & 1\\
\varepsilon 1& 0\end{array} \right),\quad u_s=\left(\begin{array}{cc}
1 & s \\
0 & 1 \end{array} \right) (s \in A^{\varepsilon\text{-}sym})
\]
\vspace{0.1cm}
\begin{defi}
\label{bruhatpres} Let $A$ be a unitary ring  equipped with an involution $*$.
We will say that $\SLc_*^{\varepsilon}(2,A)$ has a Bruhat-like  presentation if it is
generated by the above elements with defining relations:
\begin{enumerate}
\item $h_t h_{t'}=h_{tt'}$, $u_s u_{s'}=u_{s+s'}$;
\item $w^2= h_{\varepsilon}$; 
\item$h_t u_s= u_{tst^*}h_t$;
\item $wh_t=h_{t^{*-1}}w$;
\item$wu_{t^{-1}}wu_{-\varepsilon t}wu_{t^{-1}}=h_{-\varepsilon t},\quad t \in A^{\times}\cap A^{\varepsilon\text{-}sym}.$
\end{enumerate}
\end{defi}


Let us suppose that the involutive ring $A$ is finite and that the group $G= \SLc_*^{\varepsilon}(2, A)$ has a Bruhat-like  presentation.

\begin{defi} 

Let $M$ be a finite right $A$-module and let us consider  a bi-additive map  $\chi: M \times M \longrightarrow \mathbb{C}^{\times}$,  a linear character $\alpha \in \widehat{A}^{\times}$, a function $\gamma: A^{\varepsilon\text{-}sym} \times M \longrightarrow  \mathbb{C}^{\times} $  and a complex number $c$. We will say that $(M,\chi,\gamma,c)$ is a data for  $\SLc_*^{\varepsilon}(2, A)$  if the following properties hold:

\begin{enumerate}
\label{genweilrep}
\item For all $x,y \in M, t \in A^{\times}$,
\begin{enumerate}
\item $\chi(xt,y)=\alpha(tt^*)\chi(x,yt^*)$.
\item $\chi(y,x)=\chi(- \varepsilon x,y)$.
\item $\chi(x,y)=1$ for any $x \in M \Rightarrow y=0$.
\end{enumerate}

\item For all 
$s,s' \in A^{\varepsilon\text{-}sym}, x \in M$,
\begin{enumerate}
\item $\gamma(s+s',x)=\gamma(s,x)\gamma(s',x)$.
\item $\gamma(b,xt)=\gamma(tbt^*,x)$.
\item $\gamma(t,x+z)=\gamma(t,x)\gamma(t,z)\chi(x,zt)$.
\end{enumerate}

\item A complex number  $c$ such that $c^2 \vert M \vert=\alpha(\varepsilon)$, 
and for all $s \in A^{\varepsilon\text{-}sym} \cap A^{\times}$ the following equality holds:\\
$$\sum_{y \in M}\gamma(s,y)=\dfrac{\alpha(\varepsilon s)}{c}.$$
\end{enumerate}
\end{defi}

We present now a theorem due to Guti\'errez Frez, Pantoja and Soto-Andrade which  gives a method to construct  Weil representations for $\SLc_*^{\varepsilon}(2, A)$  when  a Bruhat-like presentation and  a data  $(M,\chi,\gamma,c)$  for the group are provided.
\begin{thm}\label{gen}
Let us assume that $\SLc_*^{\varepsilon}(2,A)$ has a Bruhat-like presentation. If  $(M,\chi,\gamma,c)$ is a data for 
the group then there exists  a (linear) representation $(\mathbb{C}^{M},\rho)$ of $\SLc_*^{\varepsilon}(2,A)$ defined on the generators by:
 \begin{enumerate}
\item $\rho_{u_b}(e_a)=\gamma(b,a)e_a$,
\item $\rho_{h_t}(e_a)=\alpha(t)e_{at^{-1}}$,
\item $\rho_w(e_a)=c\displaystyle{\sum_{b\in M}\chi(- \varepsilon a,b)e_b}$,
 \end{enumerate}
for $a\in M$, $b\in A^{\varepsilon\text{-}sym}$, $t\in A^{\times}$ and $e_a$ the   Dirac delta function at $a$, defined by $e_a(u)=1$ if $u=a$ and $e_a(u)=0$ otherwise.

This representation is called  generalized Weil representation of $\SLc_*^{\varepsilon}(2,A)$ associated to the data $(M,\chi,\gamma,c)$.
\end{thm}

\section{ The group $\uni(n,n)\left(\F_{q^2}/\F_q\right)$}

In what follows we introduce the unitary group $\uni(n,n)\left(\F_{q^2}/\F_q\right)$ attached to a unitary space over a finite field. We also show that this group can be regarded as a generalized classical group $\SLc_*^{-1}(2, A)$ for a suitable  involutive ring defined in the previous section.

\vspace{0.4 cm }

Let $K=\F_{q^2}$ be a quadratic extension of the finite field  $k=\F_q$, where $q$ is odd. 
Let us denote $\lambda\mapsto \overline{\lambda}$ the nontrivial $k$-automorphism of $K$. For $y=(y_1,y_2,...,y_l) \in K^l$, we put $y^*=(\overline{y_1},\overline{y_2},...,\overline{y_l})^{\tr}$, where $\tr$ denotes the transposition.

We consider  the  $K$-vector space $F=K^{2n}$ whose elements are  row vectors. We define  the hermitian form $h:F \times F \rightarrow K$ given by $h(x,y)=xJ_{+}y^*$, where $J_+=\left [\begin{smallmatrix}0& I_{n}\\I_{n}& 0\end{smallmatrix}\right]$. Since all nondegenerate hermitian bilinear forms over a finite field are equivalent \cite {grove}, we can see the quasi-split unitary group  $\uni(n,n)(K/k)$ as the isometry group of the form $h$.

Henceforth we will use  $A_n$ to indicate the  the full matrix ring $\m_{n}(K)$  endowed with the involution $*$ defined by $(a^*)_{ij}=\overline{a_{ji}}$. 

\begin{thm}
The quasi-split unitary group  $\uni(n,n)(K/k)$ is isomorphic to $\SLc_*^{-1}(2,A_n).$
\end{thm}
\begin{proof}
$\SLc_*^{-1}(2,A_n)$ is the group of isometries of the nondegenerate skew-hermitian form $i$ given by $i(x,y)=xJ_{-}y^*$, where $J_-=\left [\begin{smallmatrix}0& I_{n}\\-I_{n}& 0\end{smallmatrix}\right]$. According  to \cite{dieu} there is $\lambda \in K^{\times}$ such that $\lambda \cdot i$ is a nondegenerate hermitian form and therefore $\lambda \cdot i$ is equivalent to the form $h$ above. Thus $\SLc_*^{-1}(2,A_n)$ is isomorphic to  $\uni(n,n)(K/k)$.

\end{proof}

According to  \cite{manusc} Theorem 15 and the  result above, the group $\uni(n,n)(K/k)$  has a Bruhat-like presentation when $q>3$. So we will construct a data for this group in order to obtain a generalized Weil representation of  $\uni(n,n)(K/k)$.

\section{Data for  $\SLc_*^{-1}(2,A_n)$}
The goal in this section is to construct a concrete data for the group   $\uni(n,n)(K/k)=\SLc_*^{-1}(2,A_n)$. This will allow us to get a generalized Weil representation of  $\SLc_*^{-1}(2,A_n)$ by  using  the method  via presentation given in section two.\\

Let $V$ be the $K$-vector space given by $K^n$. We consider the pairing $\langle\quad,\quad\rangle:V\times V\to K$ given by:
\[
\langle u,v\rangle=uv^{*}.
\]
We observe directly that this pairing is nondegenerate. We also  see that  $\langle xa,y\rangle =\langle x,ya^*\rangle$, for any $a\in A_n^{\times}$.

Let  $M=V$ be the right  $A_n$-module with the $A_n$-action given by the right multiplication. Let us denote by $\sign$ the unique nontrivial character of $K^{\times}$ whose square is  trivial. \\
Let $\psi$ be a nontrivial character of $K^+$ such that $\psi(\lambda)=\psi(\overline{\lambda})$ for all $\lambda \in K^+.$
\vspace{0.5 cm}

We consider the   bi-additive function $\chi:M \times M \longrightarrow {\C}$  given by $\chi((x,y))= \psi(\langle x,y\rangle)$
and the linear  character of $A_n^{\times}$ given by 
$$\alpha(t)=\sign(\det(t)) \; \text{for} \; t\in A_n^{\times}.$$

We show then the following result.

\begin{prop} \label{chi}
\begin{enumerate} The map $\chi$ satisfies the conditions:
\item $\chi((xa,y))=\alpha(aa^*)\chi((x,ya^{*}))$, for any $x,y\in M$ and any $a\in A_n$. 
\item $\chi((y,x))=\chi((x,y)),$ for any $x,y\in M$.
\item $\chi$ is nondegenerate.
\end{enumerate}
\end{prop}

\begin{proof}
From definition of $\alpha$ we obtain $\alpha(aa^*)=1$ for every $a \in A_n$. For $x,y \in M$ and  $a\in A_n$, we see that,
\begin{enumerate}
\item 
\[
 \chi ((xa,y))= \psi(\langle xa,y\rangle)  
=\psi(\langle x,ya^{*}\rangle)  
= \chi ((x,ya^{*})).
\]

\item  
\[
\chi ((y,x))= \psi(\langle y,x\rangle) \\
=\psi(yx^*)=\psi((xy^{*})^*) 
=\psi(\langle x,y\rangle).
\]

\item Finally, we prove that $\chi$ is nondegenerate. Suppose that $\chi ((x,y))=1$ for all $x \in M$. 
 If $y \neq 0$, we observe that  the linear functional $\langle \cdot \;, y \rangle:V \longrightarrow K$ is  nontrivial which implies it is surjective.
We consider then $\lambda \in K$ satisfying  $\psi (\lambda)\neq 1$, and $t\in M$ such that
$\lambda=\langle t,y\rangle$. Then we get
\begin{equation*}
1=\psi (\langle t,y\rangle)=\psi (\lambda)\neq 1,
\end{equation*}
which forces $y=0$. 
\end{enumerate}
\end{proof}

The corresponding properties for $\gamma$ are given by,
\begin{prop} \label{gamma} 
Let $\gamma: A_n^{\varepsilon\text{-}sym}\times M \longrightarrow {{\mathbb{C}}
^\times}$  the function given by 
$\gamma(u,x)=\psi(2^{-1}\langle xu,x\rangle)$. Then,
\begin{enumerate}
\item $\gamma(u+u^{\prime},x)=\gamma(u,x) \gamma(u^{\prime},x)$,   for  $x\in M$ and $u,u'\in A_n^{\varepsilon\text{-}sym}$. 
\item $ \gamma(u,xa)=\gamma(aua^{*},x)$, for  $x\in M$, $u\in A_n^{\varepsilon\text{-}sym}$ and $a\in A_n^{\times}$.

\item $\gamma(u,x+y)=\gamma(u,x)\gamma(u,y)\chi(x,yu)$, for  $x, y\in M$ and  $u\in A_n^{\varepsilon\text{-}sym}$. 

\end{enumerate}
\end{prop}

\begin{proof}
Let $u,u^{\prime }\in A_n^{\varepsilon\text{-}sym}$, $a\in A_n^{\times}$, $x,y \in M,$

\begin{enumerate}
\item From additivity of $\langle\quad,\quad\rangle$, we see that 
\[
  \gamma(u+u^{\prime},x) 
=\psi(2^{-1}(\langle xu+xu^{\prime},x\rangle)) \\
=\psi(2^{-1}\langle x,xu\rangle)\psi(2^{-1}\langle x,xu^{\prime}\rangle)\\
=\gamma((u,x)) \gamma((u^{\prime },x)).
\]

\item  By Proposition \ref{chi} part 2, we have 
\[
 \gamma(u,xa)= \psi(2^{-1}\langle xau,xa\rangle) 
=\psi(2^{-1}\langle xaua^{*},x\rangle) 
=\gamma((aua^{*},x)).
\]
\item From  definition we see
\[
\gamma(u,x+y)=\psi(2^{-1}(xu+yu)(x+y)^*)\\
=\gamma((u,x))\gamma((u,y))\psi(2^{-1}(xuy^*+yux^*)) \\
\]
Since $\psi$ is an additive character of $K$ and it is $*$-invariant, we obtain
\[
\psi(2^{-1}(xuy^*+yux^*))=\psi(2^{-1}xuy^*)\psi(2^{-1}yux^*)=\psi(2^{-1}xuy^*)\psi(2^{-1}(x(yu)^*)^*)=\psi(xuy^*)=\chi(x,yu).
 \]
\end{enumerate}
\end{proof}

For any invertible symmetric element $u$ in $A_n^{-1-sym}$, we define $Q_u:K^{n}\longrightarrow k$ given by $Q_u(x)=\langle xu,x\rangle$. Then we prove,


\begin{lem} \label{equiv form}
The map  $Q_u$ is a nondegenerate and nonsplit $k$-quadratic form on $K^{n}$. Furthermore,  the  forms $Q_u$, $Q_{u'}$ are equivalent for any  $u, u'\in A_n^{\times}\cap A_n^{\varepsilon\text{-}sym}$. 
\end{lem}

\begin{dem} 
 We observe directly that $Q_u(\lambda x)=\lambda^{2}Q_u(x)$ for any $\lambda \in k$ and $x \in K^{n}$.
Now, let us consider the form $B_u$ associated to $Q_u$.
$$B_u(x,y)=Q_u(x+y)-Q_u(x)-Q_u(y)= \langle x,yu\rangle+\overline{\langle x,yu\rangle}$$ 
It is clear that $B_u$ is a symmetric bilinear form.
 To prove that $B_u$ is nondegenerate,  let us suppose that $B_u(x,y)=0$ for all $x \in K^{n}$. Then, for every $x \in K^{n}$, we have $\langle x,yu\rangle \in \lbrace \lambda\Delta / \lambda \in k \rbrace$. If $y \neq 0$, then the $K$-linear functional $\langle \cdot,yu\rangle:K^{n} \rightarrow K$ is  surjective. Thus for any $t\in k$, there exists $x_0 \in K^{n} $ such that $\langle x_0,yu\rangle =t$ which is impossible  and then $y=0$.\\

Finally, let  $u$, $u^{\prime}$ be two elements in $A_n^{\times} \cap A_n^{sym}$. 
 Since all nondegenerate hermitian forms in a finite field are equivalent and  $u$ and $u^{\prime}$  are invertible hermitian matrices, there exists $j \in A_n^{\times}$ such that $juj^{*}=u^{\prime}$. Thus we obtain
$Q_{u^{\prime}}(x)=Q_u(jx)$.\\
\end{dem}
\findem


\vspace{0.5cm}
Notice that $Q_{id}$ is the quadratic form given by the sum of the field norm of each coordinate. In order to facilitate calculations, we will consider only $Q_{-id}$ and denote it by $Q$.

We keep considering the above notations and set
\[
c= \dfrac{-1}{q^{n}}.
\]
Then we should prove 
\[
\sum_{x \in K^{n}} \psi(\langle x,x\rangle)= \dfrac{\alpha(-(-id))}{c}=\dfrac{1}{c}
\]
Indeed, the above sum is a Gauss sum corresponding to this quadratic form on a vector space of dimension $q^{2n}$ over $k$. Therefore the sum is $-q^{n}$ (see \cite{soto}), which is the desired result.

\vspace{1 cm}

Propositions \ref{chi}, \ref{gamma}  and Lemma \ref{equiv form} show that:
\begin{thm}
\label{data}
There exists a linear representation  $\left(\Aut(\mathbb{C}^{M}),\rho\right)$   of $ \SLc_*^{-1}(2,A_n)\cong\uni(n,n)(K/k)$  defined on the generators as follows: 
\begin{itemize}
\item $\rho_{u_b}(e_a)=\gamma(b,a)e_a$,
\item $\rho_{h_t}(e_a)=\alpha(t)e_{at^{-1}}$,
\item $\rho_w(e_a)=c\displaystyle{\sum_{b\in M}\chi(-a,b)e_b}$,
 \end{itemize} 
 for $a\in M$, $b\in A_n^{\varepsilon\text{-}sym}$, $t\in A_n^{\times}$ and $e_a$ the Dirac  delta function at $a$, defined by $e_a(u)=1$ if $u=a$ and $e_a(u)=0$ otherwise.
This representation is called the generalized  Weil representation of $\SLc_*^{-1}(2,A)\cong \uni(n,n)\left(\F_{q^2}/\F_q\right)$ associated to the data $(M,\chi,\gamma,-1/q^n)$.
\end{thm}
\begin{proof}
We see that propositions \ref{chi}, \ref{gamma}  and lemma \ref{equiv form} proves that $(M,\chi,\gamma, c)$ is a data for the group  $ \SLc_*^{-1}(2,A_n)$. Thus according to  \cite{gpsa_1} Theorem 4.4  we finally  obtain the desired representation.
\end{proof}


\section{ Initial Decomposition}

In this section we obtain  an initial decomposition of the Weil representation $(\mathbb{C}^M, \rho)$ of 
$G=\SLc_*^{\varepsilon}(2,A)$. In order to do this, we take advantage of the fact that there is a group of intertwining operators that acts naturally in $\mathbb{C}^M$, namely the unitary group $U(\chi,\gamma)$ of the pair $(\chi,\gamma)$.
To this purpose, we lean on \cite{gpsa_1} Theorem 7.6.

At the end of this section we draw on this general decomposition for our particular case $\uni(n,n)(K/k)$.

\begin{defi}
 The unitary group $\uni(\chi,\gamma)$ of the pair $(\chi,\gamma)$ is the group consisting of all $A_n$-linear automorphism $\varphi$ of $M$ such that $\gamma(b,\varphi(x))=\gamma(b,x)$, for any $b\in A_n$ and $x\in M $.
 \end{defi}
 In what follows, we just put $\uni$ for $\uni(\chi,\gamma)$.
 
 Following the idea of \cite{gpsa_1}, if we know the structure of the group $\uni$ and the set of its irreducible representations, we can find an \textit{initial} decomposition of the Weil representation in the sense  that the obtained components are not necessarily irreducible. In the sequel,  we make this decomposition explicit.

The group $ \uni$ acts (naturally) on $\mathbb{C}^M$  via $f(\beta.x)$, where   $ \beta.x=\beta(x)$,for   $\beta \in \uni$ and $x \in M$.   It follows that $\uni$ acts  on  $Aut_{\mathbb{C}}(\mathbb{C}^M)$, explicitly we  have
\begin{align*}
&\sigma: \uni \longrightarrow Aut_{\mathbb{C}}(\mathbb{C}^M),\\
&\sigma_{\beta}(f)(x)=f(\beta^{-1}.x).
\end{align*}

Arguing as in \cite{gpsa_1} we see the natural action of $\uni$ on $\mathbb{C}^M$ commutes with the action of the Weil representation of $G=\SLc_*^{\varepsilon}(2,A)$.

\vspace{0.2cm}

Let $\widehat{\uni}$ be the set of the irreducible representations of $\uni$. 
We consider the isotypic decomposition of $\mathbb{C}^M$ with respect to $\uni$:
\[
\mathbb{C}^M \cong \bigoplus_{(V_{\pi},\pi) \in \widehat{\uni}} n_{\pi}V_{\pi}.
\]


Since $n_{\pi}=\dim_{\mathbb{C}}(\Hom_{\uni}(V_{\pi},\mathbb{C}^M)) =\dim_{\mathbb{C}}(\Hom_{\uni}(\mathbb{C}^M,V_{\pi}))$ , if we put $m_{\pi}=\dim_{\mathbb{C}}(V_{\pi})$ we can write this decomposition in the following  way:
\[
\mathbb{C}^M \cong \bigoplus_{(V_{\pi},\pi) \in \widehat{\uni}} m_{\pi}\Hom_{\uni}(\mathbb{C}^M,V_{\pi}.)
\]
For $(V_{\pi},\pi)\in \widehat{\uni}$ and $\beta \in \uni$, we denote by $\pi_{\beta }$ the map $\pi(\beta):V_{\pi} \longrightarrow V_{\pi}$.
The space $\Hom_{\uni}(\mathbb{C}^M,V_{\pi})$ is formed by linear functions $\Theta:\mathbb{C}^M \longrightarrow V_{\pi}$ such that for any $\beta \in \uni$  
\begin{equation}
\Theta \circ \sigma_{\beta} =\pi_{\beta} \circ \Theta .
\label{entre}
\end{equation}

Let us consider the Delta functions $\lbrace e_x \mid x \in M \rbrace$ and the map  $\theta:M \longrightarrow V_{\pi}$  such that $\theta(x)=\Theta(e_x)$ for all $x \in M$. Since $\sigma_{\beta}(e_x)=e_{\beta.x}$, condition (\ref{entre}) becomes:

\begin{equation}
\theta(\beta.x)=\pi_{\beta} \circ \theta(x).
\label{casi}
\end{equation}
Conversely, let $\theta:M \longrightarrow V_{\pi}$ satisfying (\ref{casi}). We extend linearly and we get a map $\Theta:\mathbb{C}^M \longrightarrow V_{\pi}$ such that (\ref{entre}) holds.

Thus, we can see the space $\Hom_{\uni}(\mathbb{C}^M,V_{\pi})$ as the function space whose elements are maps
$\theta:M \longrightarrow V_{\pi}$ such that $\theta(\beta.x)=\pi_{\beta}\circ \theta(x)$, for all $\beta \in \uni$ and $x \in M$. The group $G=\SLc_*^{\varepsilon}(2,A)$ acts on this space via the Weil representation, using the same formulae as defined in Theorem (\ref{gen}). Similarly, it is possible to define the natural action of the group $\uni$ in this space, as it is formed by functions with domain $M$.

Let $\rho$ denote the Weil action of $G$ on $\mathbb{C}^M$ and $\widehat{\rho}$ the Weil action of $G$ on $\bigoplus_{(V_{\pi},\pi) \in \widehat{\uni}} m_{\pi}\Hom_{\uni}(\mathbb{C}^M,V_{\pi})$.

Due to the definition of the Weil representation, there exist scalars $K_g(x,y)\in \mathbb{C}$ depending only on $g \in G$ and $x,y \in M$ such that for all $f \in \mathbb{C}^M$, $\displaystyle{\Lambda \in \bigoplus_{(V_{\pi},\pi) \in \widehat{\uni}} m_{\pi}\Hom_{\uni}\left(\mathbb{C}^M,V_{\pi}\right)}$ the following statements hold:

\begin{align}
\rho_g(f)&= \sum_{y \in M}K_g(\cdot,y)f(y);\\
\widehat{\rho}_g(\Lambda)&=\sum_{y \in M}K_g(\cdot,y)\Lambda(y). \label{a}
\end{align}
Hence, we get:\\

\begin{lem}
 $(\mathbb{C}^M,\rho)$ and $\displaystyle{\left(\bigoplus_{(V_{\pi},\pi) \in \widehat{\uni}} m_{\pi}\Hom_{\uni}\left(\mathbb{C}^M,V_{\pi}\right),\widehat{\rho}\right)}$ are isomorphic representations of $G$.
\end{lem}


Finally we have (see \cite{vera}).\\

\begin{prop}
The space $\Hom_{\uni}(\mathbb{C}^M,V_{\pi})$ is invariant under the Weil action of $G$.
\end{prop}

\vspace{0.5cm}

Having made the decomposition above explicit, our aim is to obtain an initial decomposition for our particular case $\SLc_*^{-1}(2,A_n)$. For this purpose it is enough to know the structure of the group $\uni$ and the set of irreducible representations, which is achieved in the next proposition.\\

 \begin{prop}
 \label{desc}
 The unitary group  $U=U(M,\chi,\gamma,c)$ is isomorphic to the center of $\uni(n,n)(K/k)$, namely, the set of $\lambda \in K$ such that $\Nor_{K/k}(\lambda)=1$ .
 \end{prop}
 \begin{proof}
 Let $\beta$ be an element in $U$. From  definition $\beta$ is $K$-linear, then  $\beta(xa)=\beta(x)a$, for all $x\in K^n$ and $a\in A_n$. So in matrix language $\beta a=a\beta$, for all $a\in A_n$. This implies that $\beta$ is a scalar transformation $\lambda \ide$ of $K^n$. Now the condition  $\gamma(b,\varphi(x))=\gamma(b,x)$ forces that $\lambda$ has norm 1.
 \end{proof}
 
\section{Compatibility of methods} 
In this last section, we address the natural question about the compatibility of methods to construct  Weil representations of unitary groups.  Specifically, we show that our generalized Weil representation of 
$\uni(n,n)\left(K/k\right)$ via presentation is equivalent to the classical one constructed by G\'erardin \cite{ge}.

Let us recall that the nondegenerate skew-hermitian form $i$ on the $2n$ dimensional $K$-vector space $F$  is given by $i(x,y)=xJ_{-}y^*$, where $J_-=\left [\begin{smallmatrix}0& I_{n}\\-I_{n}& 0\end{smallmatrix}\right]$.
We denote by $E$ the underlying $k$-vector space $F$ (as in \cite{ge}),  and consider the nondegenerate symplectic form 
$j$ on $E$ given by:
\[
j(x,y)=i(x,y)-i(y,x) \;\;\; \;\;\;x,y \in F.
\]
Then the isometry group of $i$, which is  $\SLc_*^{-1}(2,A_n) \cong \uni(n,n)(K/k)$, can be embedded into $\Sp(E,j)$.\\

Hence we prove the following

\begin{thm}
The restriction of the Weil representation of the symplectic group $\Sp(E,j)$ to the group $\SLc_*^{-1}(2,A_n) \cong \uni(n,n)(K/k)$ is the representation $(\mathbb{C}^M,\rho)$ constructed in Theorem 4.4.
\end{thm}
\begin{proof} We consider the following complete polarization of the space $(K^{2n}, i)$: $F^-$ consists of all elements $x=(x_l)\in K^{2n}$ such that $x_l=0$, for an $l>n$ and $F^+$ consists  of all elements $x=(x_l)\in K^{2n}$ such that $x_l=0$, for an $l<n$. Let $(\mathbb{C}^{F^-},W^{(F,i)})$ be the restriction of the classic Weil representation of $\Sp(E,j)$ to $\SLc_*^{-1}(2,A_n)$.
It suffices to show that $W^{(F,i)}$ coincides with $\rho$ on the generators of $\SLc_*^{-1}(2,A_n)$. To prove that, we will use Propositions 34, 35 and 36 in \cite{aubert}. 

Let us note that the module $M$ in section 4 is canonically isomorphic to $F^-$. Thus, we identify $x \in M$ with $(x,0) \in F^- $.

We take  a nontrivial character  $\varphi$ of the additive group $K^+$ such that $\varphi(-2\lambda)=\psi(\lambda)$ for all $\lambda \in K^+$ and recall that $\psi$ is a nontrivial character of $K^+$ satisfying $\psi(\lambda)=\psi(\overline{\lambda})$ for all $\lambda \in K^+$. Then
\begin{enumerate}
\item Since $h_t$ preserves $F^+$ and $F^-$ we can use Proposition 34  in \cite{aubert}  to see that 
\[
W_{h_t}^{(F,i)}(f)(x)=s({\det h_t\vert_{F^-}})f(xt).
\]
Since $x\in F^-$, the action $x\mapsto xh_t$ is just $x\mapsto xt$. Then $\det h_t\vert_{F^-}=\det(t)$, whence  $s({\det h_t\vert_{F^-}})=\alpha(t)$.
\item  The element $u_s$ acts trivially on $F^+$ and on $K^{2n}/F^+$. Thus, from  \cite{aubert}  Proposition 35 we get $$W_{u_s}^{(F,i)}(f)(x)=\varphi(j(xc(-u_s),x))f(x).$$

 We consider the  Cayley  transform  $c(-u_s)=\begin{pmatrix} 0&\frac{1}{2}s\\0&0\end{pmatrix}$ in some basis. 
Then 
 $j(xc(-u_s),x)=- \langle xs, x \rangle$ and hence $\varphi(j(xc(-u_s),x))=\psi(2^{-1}\langle xs,x\rangle)$.
\item Notice that  $w$ maps bijectively  $F^{-} $ onto $F^{+} $ and $F^{+} $ onto $F^{-} $. So from the fact that  $w^2=-1$ and Proposition 36 in  \cite{aubert} we have
\[
W_{w}^{(F,i)}(f)(x)=\gamma(1)^{\dim_k(X)}\dfrac{1}{\vert F^{-}\vert^{1/2}}\sum_{y\in F^-}\varphi(j(xw,y))f(y)
\]
We observe that $\gamma(1)^{\dim_k(X)}=1$ since $\dim_k(X)$ is even and $\gamma(1)^2=1$.
Then recalling that $\psi(\overline{\lambda})=\psi(\lambda)$ for all $\lambda \in K$ we obtain,
\[
\varphi (j(xw,y))=\varphi(i(xw,y)-i(y,xw))=\varphi(-2xy^*)=\psi(\langle x,y\rangle).
\]
Therefore both representations of $\uni(n,n)(K/k)$ are equal, and the theorem follows.


\end{enumerate}
\end{proof}
According to Proposition \ref{desc}, \cite{ge} Corollary 4.5  and the theorem above  we get:
\begin{cor}
The  components of the generalized Weil representation $(\mathbb{C}^M,\rho)$  of  $\uni(n,n)\left(K/k\right)$ given in section 5 are irreducible.
\end{cor}

\end{document}